\documentclass[a4paper,12pt]{article}

\usepackage{amsmath,amssymb,amsthm, mathrsfs}
\usepackage{latexsym}
\usepackage{graphics}
\usepackage{hyperref}
\usepackage{graphicx,psfrag}
\usepackage[bf, small]{titlesec}
\usepackage{authblk} 
\usepackage{pgfpages,tikz,pgfkeys,pgfplots} 
\usetikzlibrary{arrows,positioning,matrix,fit,backgrounds,shapes}
\usepackage{soul}

\theoremstyle{plain}
\newtheorem{Thm}{Theorem}[section]
\newtheorem{Lem}[Thm]{Lemma}
\newtheorem{Prop}[Thm]{Proposition}

\theoremstyle{definition}
\newtheorem{Defi}[Thm]{Definition}

\usepackage{standalone}
\usepackage{pgfpages,tikz,pgfkeys,pgfplots}
\usetikzlibrary{arrows,positioning,matrix,fit,backgrounds,shapes,shapes.geometric, intersections}

\usetikzlibrary{shapes.callouts,decorations.pathmorphing} 
\usetikzlibrary{shadows} 
\usepackage{calc} 
\usetikzlibrary{decorations.markings} 
\tikzstyle{vertex}=[circle, draw, inner sep=0pt, minimum size=6pt] 

\usetikzlibrary{arrows,matrix} 

\newcommand{\RR}{\mathbb{R}} 

\newcommand{\TTT}{\mathcal{T}} 

\title{A graph with the partial order competition dimension greater than five}

\author{
\textsc{Jihoon Choi}%
\footnote{Department of Mathematics Education,
Cheongju University, Cheongju 28503, Korea.
\textit{E-mail}: \texttt{jihoon@cju.ac.kr}}
\quad
\textsc{Soogang Eoh}%
\footnote{Department of Mathematics Education,
Seoul National University, Seoul 08826, Korea.
\textit{E-mail}: \texttt{mathfish@snu.ac.kr}}
\quad
\textsc{Suh-Ryung Kim}%
\footnote{Department of Mathematics Education,
Seoul National University, Seoul 08826, Korea.
\textit{E-mail}: \texttt{srkim@snu.ac.kr}}
}

\date{}

\begin{document}

\maketitle

\begin{abstract}
In this paper, we use the Erd\"{o}s-Szekeres lemma to show that there exists a graph with partial  order competition dimension greater than five.
\end{abstract}




\section{Introduction}

The \emph{competition graph} $C(D)$ of a digraph $D$ is an undirected graph which
has the same vertex set as $D$ and which has an edge $xy$
between two distinct vertices $x$ and $y$
if and only if there exists a vertex $z \in V(D)$
such that $(x,z)$ and $(y,z)$ are arcs in $D$.
The notion of a competition graph was introduced by Cohen~\cite{cohen1968interval} as a means of determining the smallest dimension of ecological phase space.
In ecology, a species is sometimes characterized by the ranges of all of the different environmental factors which define its normal healthy environment.
For example, the normal healthy environment is determined by a range of values of temperature, of light, of pH, of moisture, and so on.
If there are $n$ factors in all, and each defines an interval of values, then the corresponding region in $n$-space is a box.
This box corresponds to what is frequently called in ecology the ecological niche of the species.
Hutchinson in 1944, for example, defines the ecological niche as ``the sum of all the environmental factors acting on an organism; the niche thus defined is a region of $n$-dimensional hyper-space, comparable to the phase-space of statistical mechanics.''
For this reason, the $n$-dimensional Euclidean space defined by the $n$ factors is sometimes called an ecological phase space.

We may assume that a species $x$ \emph{is superior to} another species $y$ \emph{in a certain factor} if $x$ can survive in a harsher condition of the factor than $y$.
Given factors, we say that a species $x$ preys on $y$ if $x$ is superior to $y$ in each of the given factors.
Then, given factors, we say that two distinct species \emph{compete} if $x$ and $y$ are superior to a species $z$.

Suppose we have some independent information about when two distinct species compete.
We can then ask how many dimensions are required of an ecological phase space so that we can embed each species in this space so that two species compete if and only if the independent information tells us they should.
This question can be formulated in the following way.

Let $d$ be a positive integer.
For $x = (x_1,x_2,\ldots, x_d)$,
$y = (y_1,y_2,\ldots, y_d) \in \mathbb{R}^d$,
we write
$x \prec y$ (resp.\ $x \preceq y$)
if $x_i < y_i$ (resp.\ $x_i \le y_i$) for each $i = 1, \ldots, d$.
If $x \preceq y$ or $y \preceq x$, then we say that $x$ and $y$ are \emph{comparable} in $(\RR^d, \preceq)$.
Otherwise, we say that $x$ and $y$ are \emph{incomparable} in $(\RR^d, \preceq)$.

For a finite subset $S$ of $\mathbb{R}^d$,
let $D_S$ be the digraph defined by $V(D_S) = S$ and
$A(D_S) = \{(x,y) \mid x, y \in S, \ y \prec x \}$.
A digraph $D$ is called a \emph{$d$-partial order}
if there exists a finite subset $S$ of $\mathbb{R}^d$
such that $D$ is isomorphic to the digraph $D_S$.
By convention, the zero-dimensional Euclidean space $\RR^0$ consists of a single point $0$.
In this context, we define a digraph with exactly one vertex as a $0$-partial order.
A $2$-partial order is also called a \emph{doubly partial order}.
Cho and Kim~\cite{cho2005class} studied the competition graphs of doubly partial orders
and showed that interval graphs are exactly the
graphs having partial order competition dimensions at most two.
Several variants of competition graphs of doubly partial orders
also have been studied
(see \cite{kim2009niche, kim2014competition, lu2009two, park2011m, park2013phylogeny, wu2010dimension}).
To extend these results, Choi {\it et al.}~\cite{choi2016competition} introduced the notion of the partial order competition dimension of a graph.

\begin{Defi}
For a graph $G$,
the \emph{partial order competition dimension} $\dim_{\text{{\rm poc}}}(G)$ of $G$
is the minimum nonnegative integer $d$
such that $G$ together with sufficiently many isolated vertices
is the competition graph of some $d$-partial order $D$,
i.e.,
\[
\dim_{\text{{\rm poc}}}(G) := \min \{d \in \mathbb{Z}_{\geq 0} \mid
\exists k \in \mathbb{Z}_{\geq 0}, \exists S \subseteq \mathbb{R}^d
\text{ s.t. } G \cup I_k = C(D_S)
\},
\]
where $\mathbb{Z}_{\geq 0}$ is the set of nonnegative integers and $I_k$ is a set of $k$ isolated vertices.
\end{Defi}
\noindent They characterized the graphs with partial order competition dimensions up to two.
\begin{Prop}\label{prop:dim0}~(\cite{choi2016competition})
Let $G$ be a graph.
Then, $\dim_{\text{{\rm poc}}}(G)= 0$ if and only if $G = K_1$.
\end{Prop}

\begin{Prop}\label{prop:dim1}~(\cite{choi2016competition})
Let $G$ be a graph.
Then, $\dim_{\text{{\rm poc}}}(G)= 1$ if and only if
$G = K_{t+1}$ or
$G = K_t \cup K_1$ for some positive integer $t$.
\end{Prop}
\begin{Prop}\label{prop:dim2}~(\cite{choi2016competition})
Let $G$ be a graph.
Then, $\dim_{\text{{\rm poc}}}(G) = 2$
if and only if
$G$ is an interval graph which is neither $K_s$ nor $K_t \cup K_1$
for any positive integers $s$ and $t$.
\end{Prop}

Characterizing the graphs with partial order competition dimension at most three does not seem to be easy. To approach the problem,
Choi {\it et al.}~\cite{choi2016competition} showed that trees and cycles have partial order competition dimensions at most three.

\noindent They also obtained the following useful result.
\begin{Prop}[\cite{choi2016competition}]\label{prop:induced}
Let $G$ be a graph and let $H$ be an
induced subgraph of $G$. Then
$\dim_{\text{{\rm poc}}}(H) \leq \dim_{\text{{\rm poc}}}(G)$.
\end{Prop} 
Choi {\it et al.}~\cite{choi2017partial} showed that there are chordal graphs whose partial order competition dimensions are greater than three whereas block graphs which are diamond-free chordal graphs have partial order competition dimensions at most three.
In addition, Choi {\it et al.}~\cite{choi2018partial} studied partial order competition dimensions of bipartite graphs and planar graphs and showed that their partial order competition dimensions are at most four.

Based upon the existing results on partial order competition dimensions of graphs, it is natural to ask whether or not there exists a graph with partial order competition dimension greater than four.
As a matter of fact, finding a graph with a fairly large partial order competition dimension is interesting as it shows the existence of a complex ecosystem.
In this paper, we use the Erd\"{o}s-Szekeres lemma to show that there exists a graph with partial  order competition dimension greater than five.

\section{Order types of two points in $\mathbb{R}^d$}

In this paper, we adopt the following notations.
For a positive integer $d$, we denote the set $\{1,2,\ldots,d\}$ by $[d]$.
For a subset $S$ of $[d]$, we use the notation $\overline{S}$ instead of $[d]\setminus S$
when it is obvious that $[d]$ is the universe. 
For a point $u$ in $\RR^d$, we denote the $i$th coordinate of $u$ by $[u]_i$ so that $u = ([u]_1, [u]_2, \ldots, [u]_d)$.
Take two distinct points $u$ and $v$ in $\RR^n$.
For a nonempty subset $S$ of $[d]$, we write $u \preceq_S v$ if $[u]_i \le [v]_i$ for each $i \in S$.
In addition, we denote the point $(\min\{[u]_1,[v]_1\}, \ldots, \min\{[u]_d,[v]_d\})$  in $\RR^d$ by $\min\{u,v\}$.
Therefore $[\min\{u,v\}]_i = \min\{ [u]_i, [v]_i\}$ for any $i \in [d]$.

Suppose that $u$ and $v$ are incomparable in $(\RR^d, \preceq)$.
Then there exist $i$ and $j$ in $[d]$ such that $[u]_i < [v]_i$ and $[u]_j > [v]_j$.
Therefore, the set $T$ of indices satisfying $[u]_k < [v]_k$ is nonempty and $\overline{T}$ is also nonempty.
This guarantees the existence of a partition $\{S,\overline{S}\}$ of $[d]$ such that either $u \preceq_{S} v$ and $v \preceq_{\overline{S}} u$ or $u \preceq_{\overline{S}} v$ and $v \preceq_{S} u$.
We call such a partition $\{S,\overline{S}\}$ an \emph{order type} for $\{u,v\}$.
It is clear that there exist $S(d,2) = 2^{d-1}-1$ possible order types for $\{u,v\}$ where $S(d,2)$ is a Stirling number of the second kind.
Note that order types of $\{u, v\}$ are not unique, for example, for the two points $u = (1,2,5)$ and $v=(1,3,4)$ in $\RR^3$, $\{u,v\}$ has order types $\{\{1,2\}, \{3\}\}$ and $\{\{2\}, \{1,3\}\}$. 


\begin{Lem}\label{lem:three}
Suppose that each pair of the three points $x_1 ,x_2, x_3 \in \RR^d$ is incomparable and has  a common order type $\{S,\overline{S}\}$. 
Then, for some permutation $\sigma$ on $\{1,2,3\}$, $x_{\sigma(1)} \preceq_{S} x_{\sigma(2)} \preceq_{S} x_{\sigma(3)}$ and $x_{\sigma(1)} \succeq_{\overline{S}} x_{\sigma(2)} \succeq_{\overline{S}} x_{\sigma(3)}$.
\end{Lem}
\begin{proof}
By the hypothesis, $\preceq_{S}$ is a total order on $\{x_1,x_2,x_3\}$.
Therefore there exists a permutation $\sigma$ on $\{1,2,3\}$ such that $x_{\sigma(1)} \preceq_{S} x_{\sigma(2)} \preceq_{S} x_{\sigma(3)}$.
Since $y_1 \preceq_{S} y_2$ if and only if $y_1 \succeq_{\overline{S}} y_2$ for any pair $\{y_1, y_2\}$ of points in $\RR^d$ with the order type $\{S,\overline{S}\}$, the lemma immediately follows.
\end{proof}

\begin{Lem}\label{lem:1}
For $x,y,z \in \RR^d$ and a nonempty subset $S$ of $[d]$,
suppose that
$y \preceq_{S} x$
and $z \preceq_{\overline{S}} x$.
Then $\min\{y,z\} \preceq x$.
\end{Lem}
\begin{proof}
If $i \in S$, then $[\min \{y,z\}]_i \le [y]_i \le [x]_i$, and if $i \in \overline{S}$, then $[\min \{y,z\}]_i \le [z]_i \le [x]_i$.
Therefore $\min\{y,z\} \preceq x$.
\end{proof}

\begin{Lem}\label{lem:incomparable}
Let $D$ be a $d$-partial order and let $x$ and $y$ be nonadjacent vertices in $C(D)$.
If $x$ and $y$ are non-isolated in $C(D)$, then $x$ and $y$ are incomparable in $(\RR^d, \preceq)$.
\end{Lem}
\begin{proof}
Suppose, to the contrary, that $x$ and $y$ are comparable in $(\RR^d, \preceq)$.
Then $x \preceq y$ or $y \preceq x$.
By symmetry, we may assume $x \preceq y$.
Since $x$ is a non-isolated vertex in $C(D)$, $x$ has a neighbor $z$ in $C(D)$.
Then there exists $w \in V(D)$ such that $w$ is a common out-neighbor of $x$ and $z$ in $D$, i.e., $w \prec x$ and $w \prec z$.
By the assumption $x \preceq y$, $w \prec x$ implies $w \prec y$.
Therefore $w$ is a common out-neighbor of $x$ and $y$ in $D$.
Thus $x$ and $y$ are adjacent in $C(D)$, which is a contradiction.
\end{proof}

\section{Partial order competition dimensions of complete multipartite graphs}

We denote by $K_{m \times n}$ the complete multipartite graph $K_{m,m,\ldots,m}$ having $n$ partite sets of size $m$ vertices.
In this section, we show that $\dim_\text{poc}(K_{m \times n}) \ge 6$ if $m$ and $n$ are large enough.

The Erd\"{o}s-Szekeres lemma given in~\cite{erdos1935combinatorial} states that, for any positive integers $r$ and $s$, every sequence consisting of $rs+1$ distinct real numbers has an increasing subsequence of length $r+1$ or a decreasing subsequence of length $s+1$.
The following lemma is an immediate consequence of the Erd\"{o}s-Szekeres lemma.

\begin{Lem}[\cite{erdos1935combinatorial}]\label{lem:erdos}
Let $S$ be a subset of $\RR^2$ with $|S| = n^2+1$.
Then, in the poset $(S, \preceq)$, there exists a chain or an anti-chain of length $n+1$.
\end{Lem}

\begin{Lem}\label{lem:taking xyz}
For any integers $d$ and $t$ with $d \ge 3$ and $2 \le t \le d$, let $V$ be a subset of $\RR^d$ with $|V| = 2^{2^{d-t+1}}+1$.
Then there exist distinct $x,y,z \in V$ such that, for each $j \in  \{ 1, t, t+1, \ldots, d \}$, either $[x]_j \le [y]_j \le [z]_j$ or $[x]_j \ge [y]_j \ge [z]_j$.
\end{Lem}
\begin{proof}
Suppose that the elements of $V$ are labeled as $a_1, \ldots, a_{2^{2^{d-t+1}}+1}$ so that
\begin{equation}\label{eq:chain}
[a_1]_1 \le [a_2]_1 \le [a_3]_1 \le \cdots \le [a_{2^{2^{d-t+1}}+1}]_1.
\end{equation}
Now $\{([a_i]_1, [a_i]_t) \}_{i=1}^{2^{2^{d-t+1}}+1}$ is a sequence on $\RR^2$ of length $2^{2^{d-t+1}}+1$.
By Lemma~\ref{lem:erdos}, there exists a subset $I_t \subset [2^{2^{d-t+1}}+1]$ of cardinality $2^{2^{d-t}}+1$ such that $\{([a_i]_1, [a_i]_t) \mid i \in I_t \}$ forms a chain or an anti-chain in the poset $(\RR^2, \preceq)$.
By Lemma~\ref{lem:erdos} again, there exists a subset $I_{t+1} \subset I_t$ of cardinality $2^{2^{d-t-1}}+1$ such that $\{([a_i]_1, [a_i]_{t+1}) \mid i \in I_{t+1} \}$ forms a chain or an anti-chain in $(\RR^2, \preceq)$.
We apply Lemma~\ref{lem:erdos} repeatedly in this way to obtain
a set $I_d$ of cardinality $2^{2^0}+1 = 3$ such that $I_d \subset I_{d-1} \subset \cdots \subset I_t \subset [2^{2^{d-t+1}}+1]$ and  $\{([a_i]_1, [a_i]_{j}) \mid i \in I_j \}$ forms a chain or an anti-chain in $(\RR^2, \preceq)$ for each $j=t,t+1,\ldots,d$.
Let $I_d=\{p, q, r\}$ for positive integers $p<q<r$.
Then, by \eqref{eq:chain}, $a_p$, $a_q$, and $a_r$ are desired points in $V$.
\end{proof}

The following lemma plays a key role in proving Theorem~\ref{thm:main2}, which is our main theorem.

\begin{Lem}\label{lem:main}
For positive integers $\alpha$ and $d$ with $d \ge 3$, subsets $V_1, \ldots, V_\alpha$ of $\RR^d$, and a nonempty proper subset $S$ of $[d]$, suppose that the following are true:
\begin{itemize}
\item[(i)] $V_1, \ldots, V_\alpha$ are  mutually disjoint and $|V_i| \ge 3$ for each $i=1,\ldots,\alpha$;
\item[(ii)] either  $|S|=1$ and $\alpha \ge 2$ or $|S| = 2$ and $\alpha \ge 2^{2^{d-2}}+1$;
\item[(iii)] every pair of vertices in $V_i$ is incomparable and has an order type $\{S, \overline{S}\}$ for each $i=1,\ldots,\alpha$.
\end{itemize}
Then there exist $a \in V_i$, $b \in V_j$, $c,d \in V_k$ for positive integers $1 \le i, j, k \le \alpha$ such that $i \neq j$, $c \neq d$, and $\min\{a,b\} \preceq \min\{c,d\}$.
\end{Lem}
\begin{proof}
By (ii), we consider the two cases: $|S|=1$ and $\alpha \ge 2$; $|S| = 2$ and $\alpha \ge 2^{2^{d-2}}+1$.

\textit{Case 1. $|S|=1$ and $\alpha \ge 2$.}
Without loss of generality, we may assume $S = \{1\}$.
By (i), we may take four distinct points $x,y \in V_1$ and $z,w \in V_2$.
By (iii), $\{ \{1\}, \{2,\ldots,d\}\}$ is an order type for both $\{x,y\}$ and $\{z,w\}$.
Without loss of generality, we may assume that
\begin{equation*}\label{eq:1}
x \preceq_{\{1\}} y, \quad y \preceq_{\{2,\ldots,d\}} x, \quad z \preceq_{\{1\}} w, \quad w \preceq_{\{2,\ldots,d\}} z.
\end{equation*}
In addition, we may assume
\begin{equation*}\label{eq:2}
[x]_1 \le [z]_1 \quad \text{or equivalently} \quad x \preceq_{\{1\}} z.
\end{equation*}
Then, since $w \preceq_{\{2,\ldots,d\}} z$, $\min\{x,w\} \preceq z$ by Lemma~\ref{lem:1}.
Moreover, $\min\{x,w\} \preceq w$ by definition.
Thus $\min\{x,w\} \preceq \min\{z,w\}$.
Now $i:=1$, $j:=2$, $k:=2$, $a:=x$, $b:=w$, $c:=w$, $d:=z$ satisfy the conclusion given in the lemma statement and this completes the proof in Case 1.

\textit{Case 2. $|S| = 2$ and $\alpha \ge 2^{2^{d-2}}+1$.}
Without loss of generality, we may assume $S = \{1,2\}$.
Fix $i \in [\alpha]$.
We may take three distinct points $x_i, y_i, z_i \in V_i$ by (i).
By (iii), each pair of $x_i, y_i, z_i$ has $\{S, \overline{S}\}$ as an order type.
Then, by Lemma~\ref{lem:three}, we may assume that
\begin{equation}\label{eqn:3}
x_i \preceq_{\{1,2\}} y_i \preceq_{\{1,2\}} z_i \quad \text{and} \quad z_i \preceq_{\{3,\ldots,d\}} y_i \preceq_{\{3,\ldots,d\}} x_i.
\end{equation}

Without loss of generality, we may assume
\begin{equation}\label{eq:x-linear-1st}
[y_1]_1 \le [y_2]_1 \le [y_3]_1 \le \cdots \le [y_\alpha]_1.
\end{equation}

Suppose that there exist distinct $i, j \in [\alpha]$ such that $i<j$ and $[y_i]_2 \le [y_j]_2$.
Then, since $[y_i]_1 \le [y_j]_1$ by \eqref{eq:x-linear-1st}, $y_i \preceq_{\{1,2\}} y_j$.
In addition, $z_j \preceq_{\{3,\ldots,d\}} y_j$ by \eqref{eqn:3}.
Therefore $\min \{y_i, z_j\} \preceq y_j$ by Lemma~\ref{lem:1}, which implies $\min \{y_i, z_j\} \preceq \min \{y_j, z_j\}$.
Then $k:=j$, $a:=y_i$, $b:=z_j$, $c:=y_j$, $d:=z_j$ satisfy the lemma conclusion given in the lemma statement.

Now suppose that
\begin{equation}\label{eq:x-linear-2nd}
[y_1]_2 > [y_2]_2 > [y_3]_2 > \cdots > [y_\alpha]_2.
\end{equation}
Note that $\{y_i \mid i \in [\alpha]\}$ is a subset of $\RR^d$ with cardinality $\alpha \ge 2^{2^{d-2}}+1$.
By Lemma~\ref{lem:taking xyz}, there exist distinct $p,q,r \in [\alpha]$ such that $p<q<r$ and, for each $j=1, 3, 4, \ldots, d$, either $[y_p]_j \le [y_q]_j \le [y_r]_j$ or $[y_p]_j \ge [y_q]_j \ge [y_r]_j$, which implies
$[\min \{y_p, y_r\}]_j \le [y_q]_j$.
Thus, by \eqref{eq:x-linear-2nd}, 
\begin{equation}\label{eq:min}
[\min \{y_p, y_r\}]_j \le [y_q]_j
\end{equation}
for each $j \in [d]$.

Now one of the following is true:
\[
[x_q]_1 < [y_p]_1; \quad
[x_q]_2 < [y_r]_2; \quad
[y_p]_1 \le [x_q]_1 \text{ and } [y_r]_2 \le [x_q]_2.
\]
We first consider the subcase $[x_q]_1 < [y_p]_1$.
Since $[x_q]_2 \le [y_q]_2$ by \eqref{eqn:3} and $[y_q]_2 < [y_p]_2$ by \eqref{eq:x-linear-2nd}, we have $[x_q]_2 < [y_p]_2$.
Then, by the subcase assumption that $[x_q]_1 < [y_p]_1$,
$x_q \preceq_{\{1,2\}} y_p$.
In addition, $z_p \preceq_{\{3,\ldots,d\}} y_p$ by~\eqref{eqn:3}.
Therefore $\min \{x_q, z_p\} \preceq y_p$ by Lemma~\ref{lem:1}.
Thus $\min \{x_q, z_p\} \preceq \min \{y_p,z_p\}$.
Then $i:=q$, $j:=p$, $k:=p$, $a:=x_q$, $b:=z_p$, $c:=y_p$, $d:=z_p$ satisfy the conclusion given in the lemma statement.

Secondly, we consider the subcase $[x_q]_2 < [y_r]_2$.
Since $[x_q]_1 \le [y_q]_1$ by \eqref{eqn:3} and $[y_q]_1 \le [y_r]_1$ by \eqref{eq:x-linear-1st}, we have $[x_q]_1 \le [y_r]_1$.
Then, by the subcase assumption that $[x_q]_2 < [y_r]_2$, $x_q \preceq_{\{1,2\}} y_r$.
In addition, $z_r \preceq_{\{3,\ldots,d\}} y_r$.
Therefore $\min \{x_q, z_r\} \preceq y_r$ by Lemma~\ref{lem:1}.
Thus $\min \{x_q, z_r\} \preceq \min \{y_r,z_r\}$.
Then $i:=q$, $j:=r$, $k:=r$, $a:=x_q$, $b:=z_r$, $c:=y_r$, $d:=z_r$ satisfy the conclusion given in the lemma statement.

Finally, we consider the subcase $[y_p]_1 \le [x_q]_1$ and $[y_r]_2 \le [x_q]_2$.
Take $l \in [d]$.
If $l=1$ or $l=2$, then $[\min \{y_p, y_r\}]_l \le [x_q]_l = [\min\{x_q,y_q\}]_l$ by the subcase assumption and \eqref{eqn:3}.
If $l \ge 3$, then $[\min \{y_p, y_r\}]_l \le [y_q]_l = [\min\{x_q, y_q\}]_l$ by \eqref{eq:min} and \eqref{eqn:3}.
Therefore $\min \{y_p, y_r\} \preceq \min\{x_q,y_q\}$ by Lemma~\ref{lem:1}.
Then $i:=p$, $j:=r$, $k:=q$, $a:=y_p$, $b:=y_r$, $c:=x_q$, $d:=y_q$ satisfy the conclusion given in the lemma statement and this completes the proof.
\end{proof}

Now we present a theorem which asserts that a complete multipartite graph satisfying certain properties.

\begin{Thm}\label{thm:main1}
Let $d$ be an integer with $d \ge 3$ and $D$ be a $d$-partial order.
For $\beta := 2^{2^{d-1}}+1$ and $\gamma := {d \choose 2}2^{2^{d-1}}+d+1$, suppose that $W_1, \ldots, W_\gamma$ are mutually disjoint subsets of $V(D)$ satisfying the following conditions:
\begin{itemize}
\item[(i)] $|W_i| = \beta$ for each $i \in [\gamma]$
\item[(ii)] For every $i \in [\gamma]$, every pair in $W_i$ is incomparable and has an order type with a part of size $1$ or $2$.
\end{itemize}
Then the subgraph of $C(D)$ induced by $\bigcup_{i=1}^\gamma W_i$ cannot be the complete multipartite graph $K_{\beta \times \gamma}$ with partite sets $W_1, \ldots, W_\gamma$.
\end{Thm}
\begin{proof}
Let $G$ be the subgraph of $C(D)$ induced by $\bigcup_{i=1}^\gamma W_i$.
Suppose, to the contrary, that $G$ is the complete multipartite graph $K_{\beta \times \gamma}$ with partite sets $W_1, \ldots, W_\gamma$.
Fix $i \in [\gamma]$.
Since $|W_i| = \beta = 2^{2^{d-1}}+1$, by Lemma~\ref{lem:taking xyz}, there exist three distinct points $x_i, y_i, z_i \in W_i$ such that, for each $j \in  [d]$, either $[x_i]_j \le [y_i]_j \le [z_i]_j$ or $[x_i]_j \ge [y_i]_j \ge [z_i]_j$.
Let $V_i = \{x_i,y_i,z_i\}$.
Then every pair in $V_i$ have a common type with a part of size $1$ or $2$, so the condition (iii) in Lemma~\ref{lem:main} is satisfied.
We take one of such common order types and denote it by $\TTT_i$.
Clearly $V_1, \ldots, V_\gamma$ satisfy the condition (i) in Lemma~\ref{lem:main}.


Suppose, to the contrary, that there exist distinct $i,j \in [\gamma]$ for which $\TTT_i$ and $\TTT_j$ are equal and have a part of size $1$.
By taking $V_i$ and $V_j$ as the subsets of $\RR^d$ satisfying the conditions (i), (ii), and (iii) of Lemma~\ref{lem:main} ($V_i$ and $V_j$ satisfy the first part of the condition (ii)),
we may conclude that there exist $a_i \in V_i$, $a_j \in V_j$, $a_{k_1},a_{k_2} \in V_k$ for $k \in \{i, j\}$ such that $c \neq d$ and $\min\{a_i,a_j\} \preceq \min \{a_{k_1},a_{k_2}\}$.
Since $i \neq j$, $a_i$ and $a_j$ belong to distinct partite sets of $G$ and so $a_ia_j \in E(G)$.
Then there exists $w \in V(D)$ such that $w \prec \min\{a_i,a_j\}$ by the definition of competition graph.
Since $\min\{a_i,a_j\} \preceq \min \{a_{k_1},a_{k_2}\}$, $w \prec \min \{a_{k_1},a_{k_2}\}$.
This implies $a_{k_1}a_{k_2} \in E(G)$, which contradicts the fact that $a_{k_1}$ and $a_{k_2}$ belong to the same partite set of $G$.
Therefore, for each $l \in [d]$, there exists at most one $i \in [\gamma]$ such that $\TTT_i = \{\{l\}, [d]\setminus\{l\}\}$.
Thus, among $\TTT_1, \ldots, \TTT_\gamma$, there exist at most $d$ order types each of which has a part of size $1$.

By rearranging the elements in $\{\TTT_i \mid i \in [\gamma]\}$ if necessary, we may assume that the elements in $\{\TTT_i \mid i \in [\gamma]\}$ each of which has a part of size $1$
have the largest subscripts.
Then, by the condition (ii), every element in $\{\TTT_i \mid i \in [\gamma-d]\}$ ($\gamma-d \ge 1$ by the definition of $\gamma$) has a part of size $2$.
Since there are ${d \choose 2}$ possible order types each of which has a part of size $2$,
the pigeonhole principle guarantees that there exist at least $\lceil {(\gamma-d)}/{{d \choose 2}} \rceil = 2^{2^{d-1}}+1 =: \alpha$ elements in $\{\TTT_i \mid i \in [\gamma-d]\}$ of the same order type.
By rearranging the elements in $\{\TTT_i \mid i \in [\gamma]\}$ if necessary, we may assume that $\TTT_1, \TTT_2, \ldots, \TTT_\alpha$ are the same order type.
Then, for $V_1, \ldots, V_\alpha$, the second part of condition (ii) in Lemma~\ref{lem:main} is satisfied.
Thus, by Lemma~\ref{lem:main}, there exist $a_i \in V_i$, $a_j \in V_j$, $a_{k_1},a_{k_2} \in V_k$ for some $i,j,k \in [\alpha]$ with $i \neq j$ such that $a_{k_1} \neq a_{k_2}$ and $\min\{a_i,a_j\} \preceq \min \{a_{k_1},a_{k_2}\}$.
Then, by applying the same argument as above, we may reach the conclusion that $a_{k_1}a_{k_2} \in E(G)$, which is impossible.
\end{proof}

The following theorem is an immediate consequence of Theorem~\ref{thm:main1}.

\begin{Thm}\label{thm:main2}
For $d \in \{4,5\}$, any graph containing $K_{\beta \times \gamma}$ as an induced subgraph for $\beta := 2^{2^{d-1}}+1$, $\gamma := {d \choose 2} 2^{2^{d-1}} + d + 1$ has partial order competition dimension at least $d+1$.
\end{Thm}
\begin{proof}
By Proposition~\ref{prop:induced}, it suffices to show $\dim_{\rm poc}(K_{\beta \times \gamma}) \ge d+1$.
Suppose, to the contrary, that $\dim_{\rm poc}(K_{\beta \times \gamma}) \le d$.
Then $K_{\beta \times \gamma}$ together with some additional isolated vertices is the competition graph of a $d$-partial order $D$.
Let $W_1, \ldots, W_\gamma$ be the partite sets of $K_{\beta \times \gamma}$.
Then clearly $W_1, \ldots, W_\gamma$ are mutually disjoint subsets of $V(D)$ satisfying the condition (i) in Theorem~\ref{thm:main1}.
Furthermore, no vertex in $W_i$ is isolated and no two vertices in $W_i$ are adjacent in $K_{\beta \times \gamma}$.
Thus, by Lemma~\ref{lem:incomparable}, each pair in $W_i$ is incomparable in $(\RR^d, \preceq)$ for each $i \in [\gamma]$ and so has an order type.
Since $D$ is a $d$-partial order and $d \in \{4,5\}$, its order type has a part of size $1$ or $2$.
Therefore $W_1, \ldots, W_\gamma$ satisfy the condition (ii) in Theorem~\ref{thm:main1}.
Then, by Theorem~\ref{thm:main1}, the subgraph of $C(D)$ induced by $\bigcup_{i=1}^\gamma W_i$ cannot be $K_{\beta \times \gamma}$, which contradicts the choice of $D$.
\end{proof}


\section{Concluding Remarks}

We conjecture that, for a given positive integer $d$, $\dim_\text{poc}(K_{\beta \times \gamma}) > d$ for sufficiently large positive integers $\beta$ and $\gamma$.

\section*{Acknowledgments}

The first and third author's research was supported by
the National Research Foundation of Korea(NRF) funded by the Korea government(MEST) (No.\ NRF-2015R1A2A2A01006885, No.\ NRF-2017R1E1A1A03070489) and by the Korea government(MSIP)
(No.\ 2016R1A5A1008055).
The first and second author's research was supported by Basic Science Research Program through the National Research Foundation of Korea(NRF) funded by the Ministry of Education(NRF-2018R1D1A1B07049150).


\end{document}